\documentclass[11pt]{article}

\usepackage[top=50mm, bottom=50mm, left=50mm, right=50mm]{geometry}
\usepackage{lineno}
\usepackage{amssymb}
\usepackage{amsmath}
\usepackage{amsthm}
\usepackage{amsfonts}
\usepackage{epsfig}
\usepackage{graphicx}
\usepackage{graphics}
\usepackage{float}
\usepackage{multirow}
\usepackage{lineno}
\usepackage{fullpage}
\usepackage[normalem]{ulem} 
\usepackage{makeidx}
\usepackage{xspace}
\usepackage{wrapfig}
\usepackage{color}
\usepackage{indentfirst}

\makeindex

\theoremstyle{definition}
\newtheorem{theorem}{Theorem}
\newtheorem{defi}{Definition}

\newtheorem{rmk}{Remark}

\newtheorem{lem}{Lemma}

\numberwithin{defi}{section} 
\numberwithin{theorem}{section}  
\numberwithin{lem}{section}
\numberwithin{equation}{section} 
\numberwithin{rmk}{section}

\allowdisplaybreaks[4]

\begin{document}

\title{A fixed point result of Kannan-type for multi-valued mapping on fuzzy metric spaces \footnote{ASM Subject Classifications: 54H25, 47H10 \\ keywords: fixed point theorem, fuzzy metric space, Kannan-type contraction}}
 
\author{
{Shunya Hashimoto \footnote{Department of Mathematics, Faculty of Science, Saitama University, Saitama 338-8570, Japan, email: ess71004@mail.saitama-u.ac.jp}}  \and {Aqib Saghir \footnote{Department of Mathematics, Faculty of Science, Saitama University, Saitama 338-8570, Japan, email: saghir.a.460@ms.saitama-u.ac.jp}}
}

\date{}

\maketitle

\begin{abstract}
We prove a Kannan-type fixed point theorem for multi-valued mappings on G-complete fuzzy metric spaces.
The proof uses the Hausdorff fuzzy metric space which was introduced by Rodriguez-Lopez and Romaguera \cite{Lopez}.
\end{abstract}

\section{Introduction}

Fixed point theory holds great importance in both mathematics and applied sciences offering a wide range of applications in ensuring the existence and uniqueness of solutions in differential and integral equations \cite{{Baklouti1}, {BC}}. There was a strong need to simplify and unify these ideas and principles. In his 1906 doctoral thesis, M. Frechet \cite{Frechet} effectively tackled this matter by introducing the concept of a metric space, thereby satisfying this pressing need. Based on this idea Banach \cite{Banach} applied it to establish the renowned fixed point theorem in 1922, marking a significant progression in the evolution of various extensions of metric spaces.

In 1969, Nadler \cite{Nadler} introduced a multi-valued form of Banach's theorem for metric spaces, incorporating the Hausdorff distance. Lopez et al.\cite{Lopez} extended this concept to include fuzzy metric spaces (FMS). They explored a fuzzy Hausdorff distance on the set of compact subsets within such spaces. In this context, we employ the definition provided in \cite{Lopez} to establish a principle for multi-valued fuzzy contraction mappings. It's worth mentioning that the introduction of fuzzy sets is given by L. A. Zadeh \cite{Zadeh} in 1965 which marked a significant achievement. Fuzzy concepts have made progress in almost every theoretical and applied mathematics field. Many authors within the fields of topology and analysis have since utilized this concept extensively. Kannan \cite{Kannan} extended the Banach Contraction Principle in 1968 and derived several fixed point outcomes. After Kannan, numerous mathematicians referenced as \cite{{Ciric}, {Goebel}, {Reich}} continued to explore this field and made their own valuable contributions. Banach's contraction mappings are recognized for their continuity, whereas Kannan-type mappings are not always continuous. Additionally, it's important to highlight that Banach's contraction property does not provide a characterization of metric completeness. Indeed, Subrahmanyam \cite{Subrahmanyam} proved that a metric space is complete if and only if every Kannan-type mapping possesses a fixed point.

In 1975, Michalek and Kramosil \cite{Kramosil} further employed fuzzy sets to introduce the concept of a fuzzy metric space. Additionally, George and Veermani \cite{George} in 1994 modified the characterization of fuzzy metric spaces and introduced the Hausdorff topology of fuzzy metric space. This improvement has revealed significant implications within the realm of quantum particle physics \cite{Tanaka}. As a result, many researchers have demonstrated various fixed point findings in fuzzy metric space as given in \cite{{Diamond}, {Vetro}, {Gupta}, {Mishra}, {Rold}}. S. Romaguera \cite{Salvador} has recently introduced a concept of Kannan contraction for fuzzy metric spaces, enabling the extension of Kannan's fixed point theorem to the fuzzy metric domain. This approach also facilitates the characterization of fuzzy metric completeness through an adapted version of Subrahmanyam's criterion for metric completeness.

In this paper, we provide a multi-valued variant of the Kannan contraction in the setting of fuzzy metric spaces (FMS). We extend the notion of contraction proposed by S. Romaguera to the multi-valued mapping in the context of fuzzy Hausdorff distance. 
Unlike the proof with single-valued mapping, in the case of multi-valued mapping, the construction of a Cauchy sequence cannot be taken as an iterated sequence of mapping $T$.
Therefore, we constructed the Cauchy sequence using Lemma \ref{lem2}, which is satisfied in the Hausdorff fuzzy metric space in a collection of compact subsets.
As a result, we show a Kannan-type fixed point theorem for multi-valued mappings on G-complete fuzzy metric space.
Refer to \cite{{Engelking},{Klement}} for the notion of general topology and continuous $t$-norms.

First, we introduce the fuzzy metric space in the sense of George-Veeramani \cite{George}.

\begin{defi}
A binary operation $*:[0,1]\times [0,1]\to [0,1]$ is called a continuous $t$-norm if it
satisfies the following conditions:
\begin{enumerate}
\item[(i)] $*$ is associative and commutative,
\item[(ii)] $*$ is continuous,
\item[(iii)] $a*1=a$ for all $a\in[0,1]$,
\item[(iv)] $a*b\le c*d$ wherever $a\le c$ and $b\le d$ for all $a,b,c,d\in[0,1]$.
\end{enumerate}
\end{defi}
\begin{defi}
The triple $(X,M,*)$ is said to be a fuzzy metric space if $X$ is an arbitrary non-empty set, $*$ is a continuous $t$-norm and $M$ is a function from $X\times X\times (0,\infty)$ to $[0,1]$ such that for all $x,y,z\in X$ and $t,s>0$:
\begin{enumerate}
\item[(F1)] $M(x,y,t)>0$,
\item[(F2)] $M(x,y,t)=1$ for all $t>0$ if and only if $x=y$,
\item[(F3)] $M(x,y,t)=M(y,x,t)$,
\item[(F4)] $M(x,z,t+s)\ge M(x,y,t)*M(y,z,s)$,
\item[(F5)] $M(x,y,\cdot):(0,\infty)\to[0,1]$ is continuous.
\end{enumerate}
\end{defi}
\begin{rmk}
\label{rmk1}
For each $x,y \in X$, $M(x,y,t)$ is a non-decreasing function on $(0,\infty)$.
\end{rmk}
\begin{rmk}
\label{rmk2}
From the continuity of $M$ and Remark \ref{rmk1}, for a given $x,y\in X$ if $M(x,y,t)>1-t$ for any $t>0$, then $x=y$.
\end{rmk}
We follow the paper by Grabiec \cite{Grabiec} to define the G-Cauchy sequence and G-completeness.
\begin{defi}
Let $(X,M,*)$ be a fuzzy metric space and $\{x_n\}$ be a sequence in $X$.
\begin{enumerate}
\item The sequence $\{x_n\}$ is said to be convergent if there exists $x\in X$ such that $\displaystyle \lim_{n\to\infty}M(x_n,x,t)=1$ for $t>0$.
\item The sequence $\{x_n\}$ is said to be a G-Cauchy sequence if $\displaystyle \lim_{n\to\infty}M(x_n,x_{n+q},t)=1$ for $t>0$ and $q\in\mathbb{N}$.
\item A fuzzy metric space in which every G-Cauchy sequence is convergent is called a G-complete fuzzy metric space.
\end{enumerate}
\end{defi}
\begin{defi}
Let $A$ be a non-empty subset of a fuzzy metric space $(X,M,*)$ and $t>0$.
The fuzzy distance $\mathcal{M}$ between an element $\rho\in X$ and the subset $A\subset X$ is given by
\[ \mathcal{M}(\rho,A,t)=\sup\{ M(\rho,\mu,t) : \mu\in A\}. \]
We also define that $\mathcal{M}(\rho,A,t)=\mathcal{M}(A,\rho,t)$.
\end{defi}
\begin{rmk}
\label{rmk3}
From Remark \ref{rmk1}, for any $A\subset X$ and $\mu\in X$, $\mathcal{M}(A,\mu,t)$ is a non-decreasing function on $(0,\infty)$.
\end{rmk}
\begin{lem}
\label{lem1}
If $A\in Cl(X)$, then $\rho\in A$ if and only if $\mathcal{M}(A,\rho,t)=1$ for all $t>0$, where $Cl(X)$ is the collection of nonempty closed subsets of $X$.
\end{lem}
\begin{defi}
Let $(X,M,*)$ be a fuzzy metric space. Define a function $\Theta_{\mathcal{M}}$ on $\hat{C}_0(X)\times \hat{C}_0(X)\times (0,\infty)$ by
\[ \Theta_{\mathcal{M}}(A,B,t)=\min\{ \inf_{\rho\in A}\mathcal{M}(\rho,B,t),\inf_{\mu\in B}\mathcal{M}(A,\mu,t)\}, \]
for all $A,B\in \hat{C}_0(X)$ and $t>0$, where $\hat{C}_0(X)$ is the collection of nonempty compact subset of $X$.
The triple $(\hat{C}_0(X),\Theta_{\mathcal{M}},*)$ is called a Hausdorff fuzzy metric space.
\end{defi}
\begin{rmk}
From Remark \ref{rmk3}, for any $A,B\in \hat{C}_0(X)$, $\Theta_{\mathcal{M}}(A,B,t)$ is a non-decreasing function on $(0,\infty)$.
\end{rmk}
\begin{lem}
\label{lem2}
Let $(X,M,*)$ be a fuzzy metric space such that $(\hat{C}_0(X),\Theta_{\mathcal{M}},*)$ is a Hausdorff fuzzy metric space on $\hat{C}_0(X)$. Assume that for all $A,B\in\hat{C}_0(X)$, for each $\rho\in A$ and for $t>0$, there exists $\mu_{\rho}\in B$ so that $\mathcal{M}(\rho,B,t)=M(\rho,\mu_{\rho},t)$. Then
\[ \Theta_{\mathcal{M}}(A,B,t)\le M(\rho,\mu_{\rho},t), \]
holds.
\end{lem}
We introduce the Kannan-type fixed point theorem.
\begin{theorem}
(Kannan \cite{Kannan}) Let $(X,d)$ be a complete metric space and let $T:X\to X$ be a mapping such that there exists a constant $c\in(0,\frac{1}{2})$ satisfying, for any $x,y\in X$,
\[ d(Tx,Ty)\le c[d(x,Tx)+d(y,Ty)]. \]
Then, $T$ has a unique fixed point.
\end{theorem}
Now we define the Kannan contraction of a multi-valued mapping based on Romaguera's definition of the Kannan contraction of a single-valued mapping \cite{Salvador} as follows.
\begin{defi}
Let $(X,M,*)$ be a fuzzy metric space. We say that a multi-valued mapping $T:X\to \hat{C}_0(X)$ is a (1)-Kannan contraction on $X$ if there is a constant $c\in (0,1)$ such that for any $x,y\in X$ and $t>0$,
\begin{align}
\label{kannan}
\min\{ \mathcal{M}(x,Tx,t),\mathcal{M}(y,Ty,t)\}>1-t \Rightarrow \Theta_{\mathcal{M}}(Tx,Ty,ct)>1-ct.
\end{align}
\end{defi}
We call ($s$)-Kannan contraction if the constant can be taken in the range $(0,s)$.

\section{Main result}

\indent
In this section, we prove a Kannan-type fixed point theorem for multi-valued mappings on G-complete fuzzy metric spaces.

Recall that, given a multi-valued mapping $T : X\to \hat{C}_0(X)$, a point $z$ is said to be a fixed point of $T$ if $z\in Tz$.

\begin{theorem}
Let $(X,M,*)$ be a G-complete fuzzy metric space and $(\hat{C}_0(X),\Theta_{\mathcal{M}},*)$ be a Hausdorff fuzzy metric space. Let $T : X\to \hat{C}_0(X)$ be a multi-valued (1)-Kannan contraction mapping, then $T$ has a fixed point.
\end{theorem}
\begin{proof}
Take any $x_0\in X$. Let $x_1\in X$ such that $x_1\in Tx_0$.
By Lemma \ref{lem2}, we can choose $x_2\in Tx_1$ such that for all $t>0$,
\[ M(x_1,x_2,t)\ge \Theta_{\mathcal{M}}(Tx_0,Tx_1,t). \]
Inductively, we have $x_{n+1}\in Tx_n$ satisfying
\[ M(x_n,x_{n+1},t)\ge \Theta_{\mathcal{M}}(Tx_{n-1},Tx_n,t), \ \forall n\in\mathbb{N}. \]
Fix $t_0>1$. For any $x,y\in X$ we have
\begin{align}
\label{eq1}
\mathcal{M}(x,Tx,t_0)>1-t_0, \quad \mathcal{M}(y,Ty,t_0)>1-t_0.
\end{align}
Then, from the assumption, we obtain
\[ \Theta_{\mathcal{M}}(Tx,Ty,ct_0)>1-ct_0. \]
In particular, we have
\[ \Theta_{\mathcal{M}}(Tx_0,Tx_1,ct_0)>1-ct_0, \quad \Theta_{\mathcal{M}}(Tx_1,Tx_2,ct_0)>1-ct_0. \]
Therefore, from
\begin{align*}
\mathcal{M}(x_1,Tx_1,ct_0)&\ge M(x_1,x_2,ct_0)\ge \Theta_{\mathcal{M}}(Tx_0,Tx_1,ct_0)>1-ct_0, \\
\mathcal{M}(x_2,Tx_2,ct_0)&\ge M(x_2,x_3,ct_0)\ge \Theta_{\mathcal{M}}(Tx_1,Tx_2,ct_0)>1-ct_0,
\end{align*}
and the assumption, we obtain
\[ \Theta_{\mathcal{M}}(Tx_1,Tx_2,c^2t_0)>1-c^2t_0. \]
Similarly,
\[ \Theta_{\mathcal{M}}(Tx_2,Tx_3,c^2t_0)>1-c^2t_0. \]
So, we have
\begin{align*}
\mathcal{M}(x_2,Tx_2,c^2t_0)&\ge M(x_2,x_3,c^2t_0)\ge \Theta_{\mathcal{M}}(Tx_1,Tx_2,c^2t_0)>1-c^2t_0, \\
\mathcal{M}(x_3,Tx_3,c^2t_0)&\ge M(x_3,x_4,c^2t_0)\ge \Theta_{\mathcal{M}}(Tx_2,Tx_3,c^2t_0)>1-c^2t_0.
\end{align*}
By repeating $n$ times we obtain
\[ M(x_n,x_{n+1},c^nt_0)>1-c^nt_0. \]
Here, given $t>0$, there is $n(t)\in \mathbb{N}$ such that $c^nt_0<t$ for all $n\ge n(t)$.
Therefore, considering $\frac{t}{q}$ as $t$ for $q\in\mathbb{N}$, we have
\begin{align*}
M(x_n,x_{n+q},t)&\ge M\left(x_n,x_{n+1},\frac{t}{q}\right)*M\left(x_{n+1},x_{n+2},\frac{t}{q}\right) \\
& \qquad *\cdots *M\left( x_{n+q-1},x_{n+q},\frac{t}{q}\right) \\
&\ge M(x_n,x_{n+1},c^nt_0)*M(x_{n+1},x_{n+2},c^{n+1}t_0) \\
& \qquad *\cdots *M(x_{n+q-1},x_{n+q},c^{n+q-1}t_0) \\
&>(1-c^{n}t_0)*(1-c^{n+1}t_0)*\cdots *(1-c^{n+q-1}t_0),
\end{align*}
for any $n\ge n(\frac{t}{q})$.
So, taking the limit as $n\to\infty$, $(x_n)_{n\in\mathbb{N}}$ is a $G$-Cauchy sequence in $(X,M,*)$.
Then, there is $z\in X$ such that the sequence $(x_n)_{n\in\mathbb{N}}$ converges to $z$.

Next, we prove that $z$ is a fixed point of $T$.
Fix $r,s>0$ such that $c<s<r<1$.

First, we show
\begin{align}
\label{cl1}
\mathcal{M}(z,Tz,r^kt_0)\ge 1-r^kt_0,
\end{align}
for any $k\in\mathbb{N}$. We can assume $r^kt_0\le 1$ for all $k\in\mathbb{N}$ since the opposite case gives \eqref{cl1} obviously.

For each $k\in\mathbb{N}$ we define
\[ A_{k,r,s}:=\{ \varepsilon\in(0,1) : \varepsilon+sr^{k-1}t_0<r^kt_0\}. \]
To show \eqref{cl1} by induction, let $k=1$.
Then, by (1.2) we have
\[ \mathcal{M}(z,Tz,t_0)>1-t_0, \quad \mathcal{M}(x_n,Tx_n,t_0)>1-t_0. \]
So, by the definition of $\Theta_{\mathcal{M}}$ and condition \eqref{kannan}, we obtain
\begin{align*}
\mathcal{M}(Tz,x_{n+1},st_0)&\ge \mathcal{M}(Tz,x_{n+1},ct_0)\ge \inf_{\rho\in Tx_{n}}\mathcal{M}(Tz,\rho,ct_0)\ge \Theta_{\mathcal{M}}(Tz,Tx_n,ct_0) \\
&>1-ct_0>1-st_0,
\end{align*}
for any $n\in \mathbb{N}\cup\{0\}$.
In particular, since $Tz$ is compact, there exists some $\rho\in X$ such that
\[ M(\rho,x_n,st_0)=\mathcal{M}(Tz,x_n,st_0)>1-st_0. \]
Since $(x_n)_{n\in\mathbb{N}}$ converges to $z$, for any $\varepsilon\in A_{1,r,s}$, there exists $n_{\varepsilon}\in\mathbb{N}$ such that $M(z,x_n,\varepsilon)>1-\varepsilon$ for any $n\ge n_{\varepsilon}$.
Therefore, we obtain
\begin{align*}
\mathcal{M}(z,Tz,rt_0)&\ge M(z,\rho,rt_0)\ge M(z,x_n,\varepsilon)*M(x_n,\rho,st_0) \\
&\ge (1-\varepsilon)*(1-st_0)\ge (1-\varepsilon)*(1-rt_0).
\end{align*}
If we take the limit as $\varepsilon\to 0$, then by continuity of $*$ we have
\[ \mathcal{M}(z,Tz,rt_0)\ge 1-rt_0. \]
So, we have proved when $k=1$. Next, suppose that the inequality \eqref{cl1} holds for $k=j$.
Then, we will show
\[ \mathcal{M}(z,Tz,r^{j+1}t_0)\ge 1-r^{j+1}t_0. \]
From the assumption of induction, we have
\[ \mathcal{M}(z,Tz,r^jt_0)>1-r^jt_0. \]
Also, since $(x_n)_{n\in\mathbb{N}}$ is a G-Cauchy sequence, there exists $n_j\in\mathbb{N}$ such that 
\[ M(x_n,x_{n+1},r^jt_0)>1-r^jt_0, \]
for all $n\ge n_j$.
Thus, by the definition of $\mathcal{M}$, we have
\[ \mathcal{M}(x_n,Tx_n,r^jt_0)\ge M(x_n,x_{n+1},r^jt_0)>1-r^jt_0. \]
Therefore, from condition \eqref{kannan}, we obtain
\[ \Theta_{\mathcal{M}}(Tz,Tx_n,cr^jt_0)>1-cr^jt_0, \]
for any $n\ge n_j$.
From $s>c$ and non-decreasing property, we have
\[ \Theta_{\mathcal{M}}(Tz,Tx_n,sr^jt_0)>1-sr^jt_0, \]
for any $n\ge n_j$.
Then, we get
\begin{align*}
\mathcal{M}(Tz,x_{n+1},sr^jt_0)&\ge \inf_{\rho\in Tx_{n}}\mathcal{M}(Tz,\rho,sr^jt_0)\ge \Theta_{\mathcal{M}}(Tz,Tx_n,sr^jt_0) \\
&>1-sr^jt_0.
\end{align*}
In particular, since $Tz$ is compact, there exists some $\rho\in X$ such that
\[ M(\rho,x_n,sr^jt_0)=\mathcal{M}(Tz,x_n,sr^jt_0)>1-sr^jt_0. \]
Now let $\varepsilon\in A_{j+1,r,s}$. Then $\varepsilon+sr^jt_0<r^{j+1}t_0$, and there exists $n_{\varepsilon}>n_j$ such that $M(z,x_{n_{\varepsilon}},\varepsilon)>1-\varepsilon$. Therefore,
\begin{align*}
\mathcal{M}(z,Tz,r^{j+1}t_0)&\ge M(z,\rho,r^{j+1}t_0)\ge M(z,x_{n_{\varepsilon}},\varepsilon)*M(x_{n_{\varepsilon}},\rho,sr^jt_0) \\
&\ge (1-\varepsilon)*(1-sr^jt_0)\ge (1-\varepsilon)*(1-r^{j+1}t_0).
\end{align*}
If we take the limit as $\varepsilon\to 0$, then by continuity of $*$ we have
\[ \mathcal{M}(z,Tz,r^{j+1}t_0)\ge 1-r^{j+1}t_0. \]
So, the inequality \eqref{cl1} holds.

Now, given $t>0$, since there exists $k\in\mathbb{N}$ such that $r^kt_0<t$, we have
\[ \mathcal{M}(z,Tz,t)\ge \mathcal{M}(z,Tz,r^kt_0)>1-r^kt_0>1-t. \]
By Lemma \ref{lem1}, we obtain $z\in Tz$. This completes the proof.
\end{proof}

\section*{Declarations}
\subsection*{Conflicts of interests}
The authors declare that there is no conflict of interest regarding the publication of this paper.
\subsection*{Data Availability Statements}
Data sharing is not applicable to this article as no datasets were generated or analyzed during the current study.

\end{document}